\documentclass[12pt]{article}   	% use "amsart" instead of "article" for ACSLaTeX format
\usepackage{geometry}               % See geometry.pdf to learn the layout options. There are lots.
\usepackage{graphicx}				% Use pdf, png, jpg, or epsÂ§ with pdflatex; use eps in DVI mode
\usepackage{amssymb, amsthm, amsmath, mathtools}
\usepackage{array, caption, csquotes, enumitem, float, makecell, physics, subcaption}
\MakeOuterQuote{"}
\newtheorem*{theorem*}{Theorem}

\newtheorem{theorem}{Theorem}[section]
\newtheorem{proposition}[theorem]{Proposition}
\newtheorem{lemma}[theorem]{Lemma}
\newtheorem{remark}[theorem]{Remark}

\newtheorem*{conjecture*}{Conjecture}

\newcommand{\mbf}{\mathbf}
\newcommand{\mbb}{\mathbb}

\newcommand{\R}{\mbb{R}}
\newcommand{\C}{\mbb{C}}

\title{Multi-Sink Solutions to\\ the Self-Similar Euler Equations}
\author{Hyungjun Choi, Matei P. Coiculescu}
\begin{document}
\maketitle
\begin{abstract}
We construct examples and provide a classification of self-similar solutions to the two-dimensional incompressible Euler equations whose pseudo-velocity fields possess more than one stagnation point. These solutions are also homogeneous steady states of the Euler equations. In contrast, we prove that any homogeneous self-similar solution with bounded vorticity away from the origin necessarily admits only a single stagnation point, located at the origin. The solutions we construct develop velocity cusps along rays from the origin, and this allows for additional stagnation points of the pseudo-velocity field.
\end{abstract}
\section{Introduction}
We consider the two-dimensional incompressible Euler equations:
\begin{equation}
\label{eq:EULER}
\left\{\begin{array}{r}
\partial_t u+ (u\cdot \nabla)u + \nabla p = 0,\\
u(t=0) = u_0.
\end{array} \right.
\end{equation}
One solves for the velocity vector field $u(x,t)$ and scalar pressure field $p(x,t)$ starting from initial data $u_0(x)$. The incompressible Euler equations in two-dimensions model the motion of inviscid fluids in thin plates. Due to a transport structure of the equations, the vorticity, $\textrm{curl }u$, is transported by the velocity $u$, which facilitates the existence theory for weak solutions starting from quite irregular initial data (for example, the existence for $L^1 \cap L^p$ vorticity was proven in \cite{DM}). On the other hand, uniqueness of solutions is only known when the vorticity is bounded, and non-uniqueness is only known (by convex integration arguments) due to a lack of energy conservation for $C^{1/3^-}$ velocity. More recently, the authors of \cite{BCK} used convex integration to prove non-uniqueness for $L^{1+}$ vorticity data on $\mbb{T}^2$, but these are not solutions in the usual distributional sense. The question whether unbounded and integrable ($L^1 \cap L^p$) vorticity data guarantees the uniqueness of weak solutions is open and referred to as the "Yudovich Problem".

Besides convex integration, another promising strategy for proving that there exist non-unique solutions to the Euler equations proceeds via the construction of self-similar solutions. We say that a vorticity profile $\Omega$ is a {\textit{self-similar profile}} to the incompressible Euler equations if and only if $\Omega$ satisfies the following system of equations:
\begin{equation} \label{eq:ssEE} \left\{ \begin{array}{r}
    (\nabla^\perp \Psi - \frac{1}{\alpha} \xi) \cdot \nabla_\xi \Omega = \Omega, \\
    \Delta \Psi = \Omega.
\end{array} \right. \end{equation}
Here we denoted $\nabla^\perp f = (-\partial_y f, \partial_x f)$, so the corresponding stream function profile is $\Psi$, the velocity profile is $U = \nabla^\perp \Psi$, and we have $\textrm{curl } U = \Omega$. If $\Omega$ is a self-similar profile solving Equation \eqref{eq:ssEE}, then 
\begin{equation*}
\omega(x,t):= t^{-1} \Omega(x t^{-1/\alpha}),
\end{equation*}
is a time-dependent self-similar solution of Equation (\ref{eq:EULER}). The parameter $\alpha \in (0,2)$ determines the scaling of the solution. Let us consider polar coordinates $(\rho,\theta)$ in the self-similar variables. Suppose that the profile $\Omega$ has the following asymptotic behavior at $\rho\to \infty$ (corresponding to $|\xi|\to \infty$):
$$\lim_{\rho\to\infty}\frac{\Omega(\rho,\theta)}{\rho^{-\alpha}} = g(\theta),$$
where $g(\theta)$ is a function of the polar angle only. Then we observe:
$$\lim_{t\to 0} \omega(x,t) = \lim_{t\to 0} \frac{t^{-1}\Omega(xt^{-1/\alpha})|xt^{-1/\alpha}|^{-\alpha}}{|xt^{-1/\alpha}|^{-\alpha}}= g(\theta)|x|^{-\alpha}.$$
Therefore, any self-similar solution to the Euler equations should arise from homogeneous initial vorticity data. Moreover, the initial data should be determined by the asymptotics of the self-similar profile at infinity.

Bifurcation in the self-similar variables is believed to be a plausible non-uniqueness mechanism for the Euler equations, especially since Vishik used a similar scenario to prove the non-uniqueness of solutions with time-integrable forcing in \cite{V1} and \cite{V2} (a result further refined by the authors of \cite{ABCDGJK}). We now describe the bifurcation behavior necessary for non-unique solutions to arise. If one can construct two different self-similar profiles with the same "data at infinity" in the self-similar variables, then one has proven the existence of two different self-similar solutions to the Euler equations from the same homogeneous data.

However, we remark that addressing the $L^1 \cap L^p$ Yudovich problem using this approach requires truncating the $(-\alpha)$-homogeneous vorticity data away from infinity, which seems quite difficult to resolve. Indeed, overcoming this difficulty appears to require a deep understanding of the linearization of the Euler equations in the space $L^{2/\alpha, \infty}$. We refer the reader to \cite{BC} for further discussions on this point. 

We have argued the importance of self-similar solutions to the Euler equations for the question of non-uniqueness. Homogeneous solutions to the Euler equations, like the power-law vortex with vorticity $\Omega(\xi) = |\xi|^{-\alpha}$, are self-similar solutions to the Euler equations that are also stationary in the "physical coordinates". Most of the qualitative properties of stationary homogeneous solutions to the two-dimensional Euler equations have been catalogued in the work \cite{LS} of Luo and Shvydkoy.

Perhaps one of the most intuitive and numerically validated mechanisms for non-uniqueness is the following bifurcation behavior: from an initial data of vorticity that is homogeneous, supported on two antipodal regions in angle, and symmetric with respect to rotation by $\pi$, either a single vortex spiral forms at the center of symmetry or two vortex spirals form with centers away from the center of symmetry. This bifurcation behavior was claimed to be observed numerically by the authors of \cite{BS} and discussed from the "pen and paper" point of view in \cite{BM}.

The single spiral solution has been studied with mathematical rigor by several authors. Elling first proved in \cite{E} the existence of self-similar solutions to the Euler equations from data that are small perturbations of the power-law vortex and $m$-fold rotationally symmetric with large $m$. His result was later refined by the authors of \cite{C} and \cite{SWZ}. While the results in \cite{E}, \cite{E2}, \cite{SWZ}, and \cite{C} are technically impressive, they leave the question of how to approach the Yudovich problem completely open. In particular, the non-uniqueness scenario proposed in \cite{BS} and \cite{BM} requires the existence of a self-similar solution to the Euler equations whose pseudo-velocity has multiple stagnation points. On the other hand, all the previous constructions of self-similar spirals essentially depended on small perturbations of the inviscid power-law vortex. Consequently, all the solutions they construct have self-similar {\textit{pseudo-velocity}} $U- \frac{\xi}{\alpha}$ with a single spiral stagnation point.

We believe it is worthwhile for the non-uniqueness program to discover self-similar profiles whose pseudo-velocities have more than one stagnation point. To this end, we construct and classify homogeneous self-similar profiles with pseudo-velocity fields having multiple sinks away from the origin. We call  these self-similar profiles \textbf{multi-sink solutions}. As a first example of a multi-sink solution, we have the following consequence of Theorem \ref{thm:MAIN2} and Theorem \ref{thm:MAIN3} that we prove below.

\begin{theorem*}
\label{thm:MAIN}
    For any scaling parameter $\alpha \in (0,1)$, there exist $(-\alpha)$-homogeneous self-similar vorticity profiles $\Omega$ for which the corresponding pseudo-velocity field $U-\frac{\xi}{\alpha}$ has multiple sink stagnation points away from the origin and a saddle point at the origin. In particular, up to isometries of the plane, there exists a unique such profile that has exactly two sinks and is $2$-fold rotationally symmetric.
\end{theorem*}

The construction of multi-sink solutions to the incompressible two-dimensional Euler equations is based on gluing together solutions that are locally defined on sectors. There are many ways to perform such gluings with the hope of constructing multi-sink solutions. In the case of $\alpha \in [1/2,1)$, we in fact provide a complete classification of multi-sink solutions. Taken together with the classification done in \cite{LS} of homogeneous steady states that are not obtained by gluing, our Theorem \ref{thm:MAIN2}, completes the classification of all $(-\alpha)$-homogeneous steady states of the two-dimensional Euler equations when $\alpha \in [1/2,1)$. In the case when $\alpha \in (0,1/2)$, we are also able to construct many examples of multi-sink solutions, and a finer analysis of certain "transcendental elliptic integrals" used in the construction would probably lead to a complete classification in this case as well.

Our multi-sink solutions have velocity that is only H\"{o}lder continuous, due to the gluing procedure; however, this is not a peculiarity of our construction, but rather a consequence of the additional stagnation points. In particular, we prove in Proposition \ref{IRREGULAR} below that any homogeneous self-similar solution to the two-dimensional incompressible Euler equations whose pseudo-velocity has more than one stagnation point has vorticity that is discontinuous and unbounded along an infinite ray.

We particularly note the existence of two distinct multi-sink solutions that are $2$-fold rotationally symmetric and possess two sinks away from the origin. The first example, which is additionally $4$-fold rotationally symmetric, was utilized as a barrier function in a technical argument by the authors of \cite{EH}. The second example is the one highlighted in the last sentence of our Theorem, and we refer to this as the \textbf{two-sink solution}. This solution is of special interest since it appears to converge to the power-law shear flow with velocity $(y |y|^{-\alpha},0)$ at a quadratic rate as $\alpha \to 0^+$. In this sense, the two-sink solution can be viewed as bifurcating from the power-law shear flow. We provide a combination of numerical and rigorous analytical evidence for this convergence in Section \ref{sec:SHEAR}.

The pseudo-velocity field of the two-sink solution is pictured in Figure \ref{fig:AFAR} and Figure \ref{fig:CLOSE}. Keeping in mind the numerical experiments done in \cite{BS}, we consider our two-sink solution as a plausible candidate for a non-uniqueness mechanism. The theorem we state above, in conjunction with the work showing the existence of self-similar solutions from generic data at infinity, demonstrates the feasibility of the non-uniqueness scenario proposed by Bressan and co-authors.

\textbf{Acknowledgments.}
We thank Peter Constantin and Alexandru D.~Ionescu for helpful comments on this note. We thank Tarek M.~Elgindi and Yupei Huang for pointing out their work in \cite{EH}. We acknowledge Princeton University for its support. MPC also acknowledges the support of the National Science Foundation in the form of an NSF Graduate Research Fellowship.

\begin{figure}[H]
    \centering
    \includegraphics[width=0.75\linewidth]{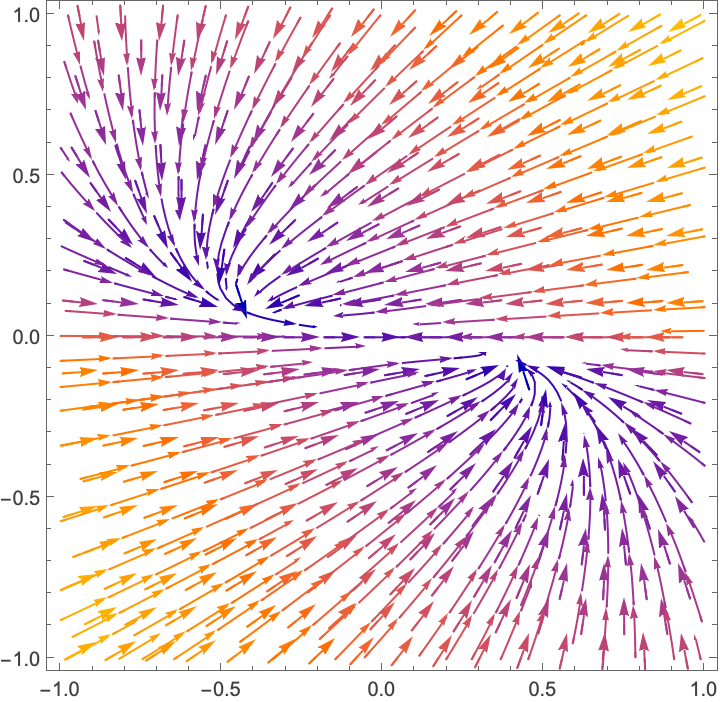}
    \caption{Vector Plot of the Pseudo-Velocity of the Two-Sink Solution when $\alpha = \frac{3}{4}$ on $[-1,1]^2$ }
    \label{fig:AFAR}
\end{figure}

\begin{figure}[H]
    \centering
    \includegraphics[width=0.75\linewidth]{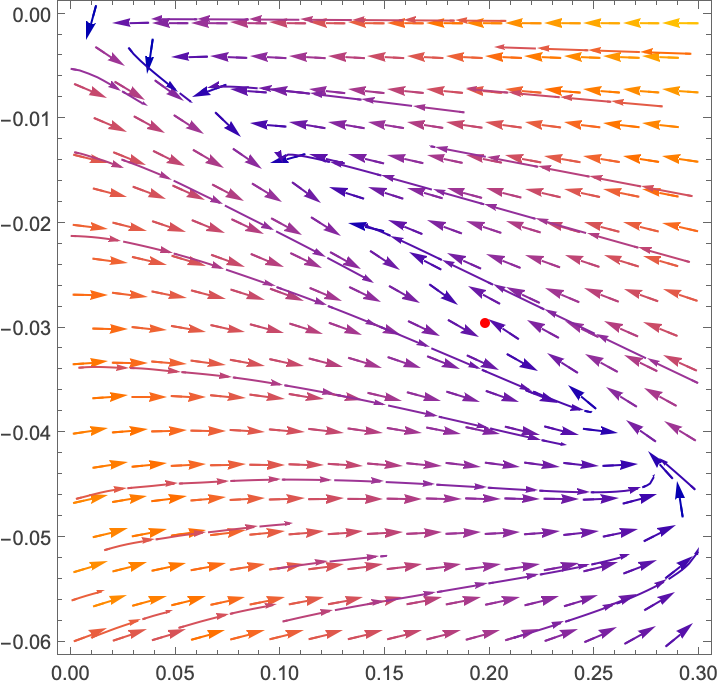}
    \caption{Vector Plot of the Pseudo-Velocity of the Two-Sink Solution when $\alpha = \frac{3}{4}$ on $[0, 0.3] \times [-0.06, 0]$. The red dot is the location of a stagnation point.}
    \label{fig:CLOSE}
\end{figure}

\section{Homogeneous Solutions of the Euler Equations and a Hamiltonian System}

Consider a homogeneous solution $u$ of the steady two-dimensional Euler equations. In other words, for some scalar pressure $p$, the vector field $u$ satisfies the system
\[u \cdot \nabla u + \nabla p = 0\]
$$\textrm{div } u=0,$$
and has the form $u = \nabla^\perp \varphi$, $\varphi = r^{\lambda} \psi(\theta)$, and $p = r^{2\lambda - 2} P(\theta)$. Then, $u = r^{\lambda - 1} (-\psi' \mbf{e}_r + \lambda \psi \mbf{e}_\theta)$ and putting this into the steady two-dimensional Euler equations gives
\begin{align}
    - \lambda \psi \psi'' + (\lambda - 1)(\psi')^2 - \lambda^2 \psi^2 + 2(\lambda - 1)P &= 0, \label{eq:ODE} \\
    P' &= 0. \notag
\end{align}
If $\lambda > 0$, $P$ is a constant, and $\psi\in H^1(\mbb{T})$ solves Equation \eqref{eq:ODE}, then $u = \nabla^\perp(r^\lambda \psi(\theta)) \in L^2_{\text{loc}}$ is a steady state of the two-dimensional Euler equations. Since the solution $(u,p)$ is homogeneous, it is also a steady solution of the two-dimensional self-similar Euler equations (Equation \eqref{eq:ssEE}) with self-similar variable $\xi = x/t^{1/\alpha}$ where $\alpha = 2-\lambda$.

In general, we find homogeneous solutions $\Psi = r^{2-\alpha} \psi(\theta)$ by solving the ordinary differential equation given in Equation \eqref{eq:ODE}, which turns out to be a Hamiltonian system. This method for constructing homogeneous solutions to the two-dimensional Euler equations was used in \cite{LS} and is probably known to earlier mathematicians as well. To construct multi-sink solutions, we find homogeneous solutions with pseudo-velocity vector field $\nabla^\perp \Psi - \frac{1}{\alpha} \xi$ having several sink stagnation points.

\section{Construction of Multi-Sink Solutions}
In this section, we build upon the preliminary results on homogeneous solutions presented in the paper \cite{LS} of Luo and Shvdkoy, and we construct multi-sink solutions. While Luo and Shvdkoy perform a complete classification of local solutions to the Hamiltonian system, their paper omits the construction of any particular non-trivial $2\pi$-periodic function $\psi(\theta)$. Our construction requires a fine asymptotic analysis of the period functions associated to the Hamiltonian system that will occupy us for the rest of this note.

\subsection{Local Solutions}
We rewrite Equation \eqref{eq:ODE} as
\begin{equation} \label{eq:2.1}
    \psi'' = \frac{\lambda-1}{\lambda} \frac{(\psi')^2 + 2P}{\psi} - \lambda \psi.
\end{equation}
If $\psi$ is sign-definite on an interval $I = (a,b)$ and solves Equation \eqref{eq:2.1}, then it is smooth in $I$. When a local solution is extended to its maximal interval of existence, $\psi$ vanishes at the end points (unless an end point is $\pm\infty$). If $\psi \in C^\infty((-a,a))$ is a maximally extended local solution, then $-\psi$ and $\psi(-\theta)$ are also local solutions. Thus, every negative local solution is obtained by multiplying a positive solution by $(-1)$ and vice versa. Also, every local solution is even with respect to the midpoint of the maximal interval of existence.

Equation \eqref{eq:ODE}, equivalently Equation \eqref{eq:2.1}, has a first integral that originates from the Bernoulli function $|u|^2 + 2p$. Since the steady two-dimensional Euler equations can be re-written as $\omega u^\perp + \nabla(|u|^2 + 2p) = 0$, we obtain
\begin{equation} \label{eq:2.2}
u \cdot \nabla (|u|^2 + 2p) = 0.
\end{equation}
Assume that $\psi$ is a positive local solution. Then, if we in addition assume the homogeneous ansatz, integrating Equation \eqref{eq:2.2} yields
\begin{equation} \label{eq:2.3}
B = (2P + (\psi')^2 + \lambda^2 \psi^2)\psi^{\frac{2}{\lambda} - 2}.
\end{equation}
Equation \eqref{eq:2.1} can be rewritten using the conservation law in Equation \eqref{eq:2.3}:
\begin{equation} \label{eq:2.4}
\psi'' = \frac{(\lambda - 1)B}{\lambda} \psi^{1 - \frac{2}{\lambda}} - \lambda^2 \psi.
\end{equation}
If $(\psi, P)$ is a local solution, then $(a\psi, a^2 P)$ is also a local solution for any $a>0$. Thus, we may assume $B = 1$ or $B = -1$ without loss of generality.

From now on, we always assume that $1<\lambda<2$ (equivalently $\alpha \in (0,1)$). The streamlines of the phase portrait for the ordinary differential equation are given by the curves $(\psi, \psi')$, where 
\[2P = B \psi^{2 - \frac{2}{\lambda}} - (\psi')^2 - \lambda^2 \psi^2\]
is always satisfied. Now we have two cases depending on the sign of $B$.

\noindent\textit{Case 1.} $B = 1$\\
For $P > 0$, the streamline is a closed smooth curve in the half plane $\{\psi > 0\}$, and for $P = 0$, the streamline touches the axis $\{\psi = 0\}$ tangentially at the origin. For $P<0$, the streamline is a smooth curve connecting $(0,\pm\sqrt{-2P})$.\\

\noindent\textit{Case 2.} $B = -1$\\
In this case, we necessarily have $P < 0$ and that the streamline is a smooth curve connecting $(0,\pm\sqrt{-2P})$.

\subsection{Gluing Local Solutions}
Let $\psi_1 : [0,T_1] \to \R_{\geq0}$ and $\psi_2 : [0, T_2] \to \R_{\geq0}$ be maximal local solutions with the same value of $P < 0$. Then, we can glue two solutions $\psi_1$ and $-\psi_2$ at the endpoint to obtain a solution $\psi : [0,T_1 + T_2] \to \R$ to the original differential equation in Equation \eqref{eq:ODE}. From the construction, $\psi \in C^1$ and since
\begin{equation}
\label{eq:FIRSTORDER}
(\psi')^2 = -2P + \text{sign}(\psi) |\psi|^{2-\frac{2}{\lambda}} - \lambda^2 \psi^2,
\end{equation}
we in fact get $\psi \in C^{1,2-\frac{2}{\lambda}}$. In this way, we can glue any even number of maximally defined local solutions to obtain a $C^{1,2-\frac{2}{\lambda}}$ solution to Equation \eqref{eq:ODE}. When $1<\lambda<2$ and $P < 0$, there are two distinct, positive, and maximal local solutions, which both connect $(0,\pm\sqrt{-2P})$ in the phase portrait of the Hamiltonian system. We denote by $\psi_\pm$, the maximal local solution corresponding to $B=\pm 1$.

\begin{remark}
The solution in the case $P=0$ is a shear flow $\psi = A|\cos(\theta)|^\lambda \in C^{1,\lambda - 1}$. Thus, a solution constructed by gluing exhibits better regularity than the shear flow solution.
\end{remark}

Let $T_\pm$ be the lifespans of maximal local solutions corresponding to $B=\pm 1$. Since we have an analytic expression for the streamlines, $T_\pm$ can be written as an integral:
\begin{equation} \label{eq:PERIOD}
T_\pm = \int_{0}^1 \frac{2}{\sqrt{\lambda^2 (1-s^2) \mp x_\pm^{-2/\lambda} (1 - s^{2-2/\lambda})}}\,ds,
\end{equation}
where $x_\pm$ are positive solutions of
\[\lambda^2 x_{\pm}^2 \mp x_{\pm}^{2-\frac{2}{\lambda}} = -2P.\]
The integrals $T_\pm$ are the "transcendental elliptic integrals" we referred to in the introduction.
\begin{lemma}
    If $1 < \lambda < 2$, then the integrals $T_\pm$ are continuous monotonic functions of $P\in (-\infty, 0]$, where $T_+(P=0) = \pi$, $T_-(P=0) = 0$, and $\displaystyle\lim_{P\to -\infty} T_\pm = \frac{\pi}{\lambda}$.
\end{lemma}
\begin{proof}
Since $\lambda^2 = x_+^{- \frac{2}{\lambda}} -2P x_+^{-2}$, we observe that $x_+^{-\frac{2}{\lambda}}$ as a function of $-P\in [0, \infty)$ is continuous, monotonically decreasing, and
\[x_+^{-\frac{2}{\lambda}}(P = 0) = \lambda^2, \quad x_+^{-\frac{2}{\lambda}}(P = -\infty) = 0.\]
When $x_+^{-\frac{2}{\lambda}} \in [\lambda^2 ,\infty]$, the integral expression of $T_+$ in Equation \eqref{eq:PERIOD} is integrable, so $T_+$ is continuous and monotonically decreasing in $-P\in [0, \infty)$. In particular, the values of $T_+$ at the endpoints $P=0$ and $P=-\infty$ are evaluated to be
\[T_+(P=0) = \int_0^1 \frac{2}{\lambda \sqrt{s^{2-2/\lambda}-s^2}} \,ds = \int_0^1 \frac{2}{\sqrt{1-t^2}}\,dt = \pi,\]
and
\[T_+(P=-\infty) = \frac{2}{\lambda} \int_0^1 \frac{1}{\sqrt{1-s^2}} \,ds = \frac{\pi}{\lambda}.\]
Similarly, $x_-$ as a function of $-P\in [0, \infty)$ is continuous, monotonically increasing, and
\[x_-(P = 0) = 0, \quad x_-(P = -\infty) = \infty.\]
Thus, $T_+$ is continuous and monotonically increasing in $-P\in [0, \infty)$. Moreover, the values of $T_+$ at the endpoints $P=0$ and $P=-\infty$ are
\[ T_-(P=0) = x_-^{1/\lambda}(P=0) \int_0^1 \frac{2}{\lambda}\frac{ds}{\sqrt{1-s^2}} = 0, \]
and
\[T_- (P=-\infty) = \frac{2}{\lambda} \int_0^1 \frac{1}{\sqrt{1-s^2}} \,ds = \frac{\pi}{\lambda}. \]
\end{proof}

\begin{theorem} \label{thm:MAIN2}
    If $1 < \lambda < 2$, then a $2\pi$-periodic solution to Equation \eqref{eq:ODE} can be obtained by gluing any of the following combinations in any order with alternating signs:
\begin{enumerate}[label=(\roman*)]
    \item $2k$ copies of $\psi_-$, for $k \geq 2$.
    \item One copy of $\psi_+$ and $2k-1$ copies of $\psi_-$, for $k \geq 2$.
    \item Two copies of $\psi_+$.
    \item Two copies of $\psi_+$ and two copies of $\psi_-$.
\end{enumerate}
    In particular, these possibilities together with (v) below, constitutes the complete list of possible gluing combinations for $1 < \lambda \leq \frac{3}{2}$.
\begin{enumerate}[resume, label=(\roman*)]
    \item Two copies of $\psi_+$ and $2k$ copies of $\psi_-$, for $k\geq 2$.
\end{enumerate}
\end{theorem}
\begin{remark}
(a) In the case (iii), we necessarily have that $P=0$ since $T_+(P=0) = \pi$ and $T_+$ is otherwise smaller than $\pi$. The $P=0$ solution to Equation \eqref{eq:ODE} is the homogeneous shear flow solution, either $\Psi = |y|^\lambda$ or $y |y|^{\lambda - 1}$.

(b) The case (iv), if we glue in the order $\psi_+, -\psi_-, \psi_+, -\psi_-$, yields a $2$-fold symmetric homogeneous solution to the two-dimensional Euler equations, which we refer to as the \textbf{two-sink solution}.

(c) In the case (i) with $k=2$, the resulting solution gives an exactly $4$-fold rotationally symmetric homogeneous steady solution to the two-dimensional Euler equations. This case was observed by Elgindi and Huang in the appendix of \cite{EH}.
\end{remark}

\noindent\textit{Proof of Theorem \ref{thm:MAIN2}.}
We will show in every case that there exists $P < 0$ such that sum of the lifespans equal $2\pi$.\\

\noindent \textit{Case (i)}. Since $T_-(P=0) = 0$ and $T_-(P=-\infty) = \frac{\pi}{\lambda} > \frac{\pi}{2}$, there exists $P < 0$ such that $T_-(P) = \frac{\pi}{k}$ for any $k\geq 2$.\\

\noindent \textit{Case (ii)}. Notice that $T_+(P=0) + (2k-1) T_-(P=0) = \pi$ and $T_+(P=-\infty) + (2k-1) T_-(P=0) = \frac{2k\pi}{\lambda} > 2\pi$. Thus, there exists $P < 0$ such that $T_+(P=0) + (2k-1) T_-(P=0) = 2\pi$ for any $k\geq 2$.\\

\noindent \textit{Case (iii)}. Simply $T_+(P=0) = \pi$.\\

\noindent \textit{Case (iv)}. Notice that $T_+(P=0) + T_-(P=0) = \pi$ and $T_+(P=-\infty) + T_-(P=-\infty) = \frac{2\pi}{\lambda} > \pi$. We prove in the next section that $T_+ + T_-$ is a decreasing function of $-P$ near $P=0$. This shows that there exists $P < 0$ such that $T_+(P=0) + T_-(P=0) = \pi$.\\

\noindent \textit{When $1 < \lambda \leq \frac{3}{2}$}. We see that $T_+ > \frac{\pi}{\lambda}\geq \frac{2\pi}{3}$. Thus, there cannot be more than two copies $\psi_+$ if we are to obtain a $2\pi$-periodic solution. In addition, since $T_\pm < \pi$ for $P < 0$, it is impossible to have only two copies of $\psi_-$, nor is it possible to have one copy of $\psi_+$ and $\psi_-$ each. It follows that the cases (i), (ii), (iii), (iv), and (v) are only possible combinations. Indeed, we prove in the next section that $T_+ + kT_-$ is a decreasing function of $-P$ near $P=0$ when $1 < \lambda < \frac{3}{2}$, so the case (v) is valid. \qed\\

We can now finish the proof of the theorem we stated in the introduction.
\begin{theorem}
\label{thm:MAIN3}
Suppose $P < 0$, and let $\psi$ be a $2\pi$-periodic solution to Equation \eqref{eq:ODE} obtained by gluing $2k$ maximal local solutions. Consider the corresponding homogeneous steady state $\Psi = r^\lambda \psi(\theta)$. Then, the pseudo-velocity field $\nabla^\perp \Psi - \frac{\xi}{\alpha}$ has a stagnation point at each angle $\theta$ where $\psi(\theta)$ changes sign from positive to negative. Moreover, all $k$ stagnation points other than the origin are sinks located at a distance of $\alpha^{\frac{1}{\alpha}} (-2P)^{\frac{1}{2\alpha}}$ from the origin. Finally, the origin is a saddle point.
\end{theorem}
\begin{proof}
The corresponding pseudo-velocity is:
$$ U - \frac{\xi}{2-\lambda} = r^{\lambda - 1} (-\psi' \mbf{e}_r + \lambda \psi \mbf{e}_\theta) - \frac{r}{2-\lambda} \mbf{e}_r.$$
At a stagnation point of the pseudo-velocity, we evidently require $\psi=0$ and 
$$r^{\lambda-1}\psi' +\frac{r}{2-\lambda} =0.$$
Therefore, on the angle $\theta$ with $\psi(\theta) = 0$, $\psi'(\theta) < 0$, there is a stagnation point with distance $r$ determined by the equation above. Since $|\psi'(\theta)| = \sqrt{-2P}$ given that $\psi(\theta) = 0$, we show that the distance $r = \alpha^{\frac{1}{\alpha}} (-2P)^{\frac{1}{2\alpha}}$.

Now, it is not difficult to see that the eigenvalues of $\nabla(U-\xi/(2-\lambda))$ at the two stagnation points away from the origin are precisely $-1$ and $\tfrac{-\lambda}{2-\lambda}$, so each of these stagnation points away from the origin is a sink. Lastly, the origin is a saddle since the flow is coming out (in) in the direction where $\psi' < 0 ~(>0)$.
\end{proof}

We also prove the following proposition, which shows that any homogeneous self-similar solution with more than one stagnation point is necessarily irregular.

\begin{proposition}
\label{IRREGULAR}
Let $\alpha \in (0,1)$.
    Suppose that $U(\xi)$ is a $(1-\alpha)$-homogeneous steady solution to the two-dimensional Euler equations. Suppose in addition that the corresponding pseudo-velocity 
    $$\tilde{U}(\xi):= U(\xi) - \frac{\xi}{\alpha}$$
    has a stagnation point in addition to the one at the origin $\xi=0$. Then the vorticity $\textrm{curl}~ U$ is discontinuous on a ray from the origin. 
\end{proposition}
\begin{proof}
We recall from Luo and Shvydkoy \cite{LS} that for $\alpha\in (0,1)$ (equivalently $\lambda \in (1,2)$), the only $2\pi$-periodic stream functions $\psi(\theta)$ possible correspond to the power-law vortex $\psi(\theta)=A$ for some constant $A$, the power-law shear flow $\psi(\theta) = A|\cos(\theta)|^\lambda$ for some constant $A$, and any solution obtained by gluing two local solutions to the Hamiltonian system together. 

As we observed in the proof of Theorem \ref{thm:MAIN3}, if the pseudo-velocity $\tilde{U}$ has a stagnation point for $r>0$, the it necessarily occurs at an angle $\theta$ where the stream function $\psi(\theta)=0$ and $\psi'(\theta)<0$. It is therefore not difficult to see that both the power-law vortex and the power-law shear flow are inadmissible (for one $\psi'=0$ always, for the other $\psi'=0$ whenever $\psi=0$). The only candidates remaining are obtained by gluing. As we previously observed, any gluing of local solutions yields at most $\psi \in C^{1,2-2/\lambda}$ regularity. In particular, the vorticity is discontinuous along the ray corresponding to the angle at which the stagnation points of $\tilde{U}$ occur.
\end{proof}

\subsection{Asymptotics as $P \to 0^-$ of the Periods $T_\pm$}

We finish the proof of Theorem \ref{thm:MAIN2} by showing that
\newpage
\begin{itemize}
    \item For $1 < \lambda < 2$, $T = T_+ + T_-$ decreases as $-P$ increases near $P=0$, and
    \item For $1 < \lambda \leq\frac{3}{2}$ and $k\geq 0$, $T_+ + kT_-$ decreases as $-P$ increases near $P=0$.
\end{itemize}
The following uses the theory presented in \cite{BH86}.

\begin{figure}[H]
    \centering
    \includegraphics[width=0.8\linewidth]{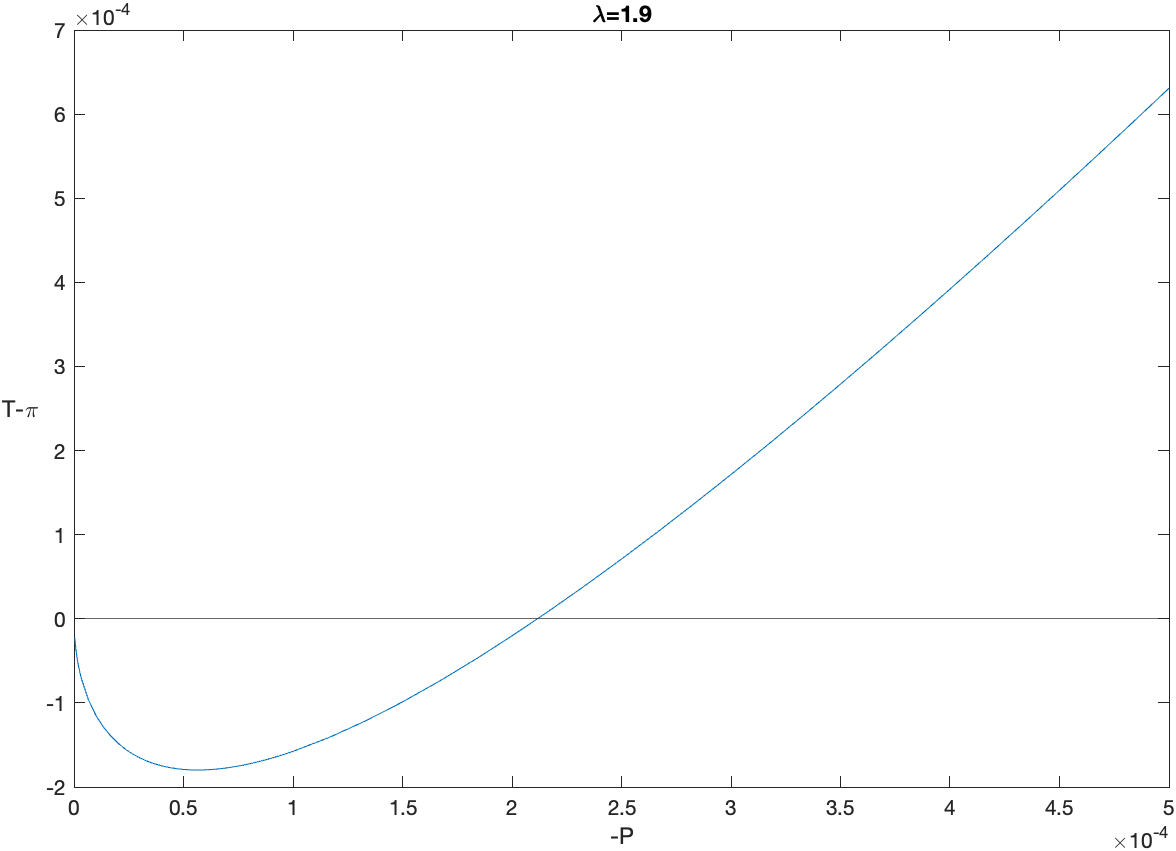}
    \caption{Graph of $T-\pi$ near $P=0$ when $\lambda = 1.9$}
    \label{fig:P}
\end{figure}

First, we consider $T_+$. Let $a = (\lambda^2 - x_+^{-2/\lambda})/\lambda^2$, then
\[T_+ = \frac{2}{\lambda} \int_0^1 \frac{ds}{\sqrt{s^{2-2/\lambda}-s^2 + a(1 - s^{2-2/\lambda})}}.\]
Since $- 2P \lambda^{2\lambda-2} (1-a)^\lambda = a$,
\[a = 2\lambda^{2\lambda - 2}|P| - 4 \lambda^{4\lambda-3} |P|^2 + O(|P|^3).\]
For $\lambda\in (1,\frac{3}{2})$, the $a$-derivative of the integrand is integrable, thus
\begin{align*}
T_+ &= \pi - \frac{2}{\lambda} \int_0^1 \frac{1-s^{2-2/\lambda}}{(s^{2-2/\lambda}-s^2)^{3/2}} \,ds \cdot a + o(a) \\
&= \pi - \int_0^1 \frac{t^{\frac{1}{2} - \lambda} - t^{-\frac{1}{2}}}{(1-t)^{\frac{3}{2}}} \,dt \cdot a + o(a) \\
&= \pi - \frac{2\sqrt{\pi} \Gamma(\lambda) \tan(\pi\lambda)}{\Gamma(\lambda - \frac{1}{2})}\cdot a + o(a).
\end{align*}
Therefore, for $\lambda\in (1,\frac{3}{2})$,
\[T_+ = \pi - \frac{4 \lambda^{2\lambda - 2} \sqrt{\pi} \Gamma(\lambda) \tan(\pi\lambda)}{\Gamma(\lambda - \frac{1}{2})} |P| + o(|P|).\]
For $\lambda \in [\frac{3}{2}, 2)$, let
\[f(s) = \frac{1}{\sqrt{s^{2-2/\lambda} - s^2}}, \quad g(t) = \sqrt{\frac{t}{1+t}}, \quad \phi(s) = \frac{s^{2-2/\lambda}-s^2}{1-s^{2-2/\lambda}}.\]
Then,
\[T_+ = \frac{2}{\lambda} \int_0^1 f(s) g(a^{-1} \phi(s)) \,ds = \frac{2}{\lambda} \int_0^{\frac{1}{\lambda-1}} F(t) g(a^{-1} t) \,dt. \]
We have the following expansions for $g, \phi$, and $F(t) = f(s)/\phi'(s)|_{s=\phi^{-1}(t)}$:
\begin{align*}
g(t) &= 1 - \frac{1}{2} t^{-1} + \frac{3}{8} t^{-2} - \frac{5}{16} t^{-3} + \ldots & \text{as}~t\to \infty, \\
\phi(s) &= s^{2- \frac{2}{\lambda}} \big(1 + s^{2-\frac{2}{\lambda}} - s^{\frac{2}{\lambda}} + s^{4-\frac{4}{\lambda}} + \ldots \big)  & \text{as}~ s\to 0^+, \\
\phi^{-1}(t) &= t^{\frac{\lambda}{2\lambda-2}} \bigg(1 - \frac{\lambda}{2\lambda-2} t + \frac{\lambda}{2\lambda-2} t^{\frac{1}{\lambda-1}} + \ldots \bigg) & \text{as}~ t\to 0^+,\\
\frac{f(s)}{\phi'(s)} &= \frac{\lambda}{2} s^{\frac{3}{\lambda}-2} \frac{(1-s^{2-2/\lambda})^2}{\lambda - 1 - \lambda s^{2/\lambda} + s^2} \frac{1}{\sqrt{1-s^{2/\lambda}}} \\ &=\frac{\lambda}{2(\lambda-1)} s^{\frac{3}{\lambda}-2} \bigg( 1 - 2 s^{2-\frac{2}{\lambda}} + \frac{3\lambda-1}{2\lambda-2} s^{\frac{2}{\lambda}} + s^{4-\frac{4}{\lambda}} + \ldots \bigg) & \text{as}~ s\to 0^+, \\
F(t) &= \frac{\lambda}{2\lambda-2} t^{\frac{3-2\lambda}{2\lambda-2}} \bigg( 1 - \frac{2\lambda-1}{2\lambda-2} t + \ldots \bigg) & \text{as}~ t\to 0^+.
\end{align*}
From the asymptotic expansion formula located in Appendix B, we have for $\lambda \in (\frac{3}{2} , 2)$,
\begin{align*}
T_+ &= \frac{2}{\lambda} M[F;1] + \frac{1}{\lambda-1} M \bigg[g;\frac{1}{2\lambda-2}\bigg] a^{\frac{1}{2\lambda-2}} - \frac{1}{\lambda} M[F;0] a + O(a^{1 + \frac{1}{2\lambda-2}}) \\ % - \frac{\lambda(2\lambda-1)}{4(\lambda-1)^2} M \bigg[g;\frac{2\lambda-1}{2\lambda-2}\bigg] a^{\frac{2\lambda-1}{2\lambda-2}} + o(a^{\frac{2\lambda-1}{2\lambda-2}})
&= \pi - \frac{2 \Gamma(\frac{2\lambda - 1}{2\lambda-2}) \Gamma(\frac{\lambda}{2\lambda - 2}) }{\sqrt{\pi}} a^{\frac{1}{2\lambda-2}} - \frac{\sqrt{\pi} \Gamma(\lambda) \tan(\lambda\pi)}{\Gamma(\lambda-\frac{1}{2})} a + O(a^{1 + \frac{1}{2\lambda-2}}).
\end{align*}
Note that
\begin{align*}
M[F;z] &= \int_0^1 f(s) \phi(s)^{z-1} \,ds = \int_0^1 \frac{(s^{2-2/\lambda}-s^2)^{z-\frac{3}{2}}}{(1-s^{2-2/\lambda})^{z-1}} \,ds \\
&= \frac{\lambda}{2} \int_0^1 \frac{w^{(\lambda-1)(z-1)-\frac{1}{2}} (1-w)^{z-\frac{3}{2}}}{(1-w^{\lambda-1})^{z-1}} \,dw.
\end{align*}
Therefore, for $\lambda\in (\frac{3}{2}, 2)$,
\[T_+ = \pi - \frac{2^{1 + \frac{1}{2\lambda - 2}} \lambda \Gamma(\frac{2\lambda - 1}{2\lambda-2}) \Gamma(\frac{\lambda}{2\lambda - 2})}{\sqrt{\pi}} |P|^{\frac{1}{2\lambda-2}} - \frac{2\sqrt{\pi} \lambda^{2\lambda - 2} \Gamma(\lambda) \tan(\lambda\pi)}{\Gamma(\lambda-\frac{1}{2})} |P| + O(|P|^{1 + \frac{1}{2\lambda-2}}).\]
For $\lambda = \frac{3}{2}$, we have a logarithmic singularity ($\lambda=3/2$ is known to be an interesting case, compare with the explicit homogeneous solutions known for this value of $\lambda$ in \cite{LS} and the Kaden spiral in \cite{K}):
\[T_+ = \pi - \frac{3}{4} a \log(a^{-1}) + O(a) = \pi - \frac{9}{4} |P| \log (|P|^{-1}) + O(|P|).\]

Now, we consider $T_-$. Let $b = x_-^{\frac{1}{\lambda}}$, then
\[T_- = 2b \int_0^1 \frac{ds}{\sqrt{\lambda^2 b^2 (1-s^2) + (1-s^{2-2/\lambda})}}.\]
Since $b^{2\lambda - 2} + \lambda^2 b^{2\lambda} = -2P$,
\[b = 2^{\frac{1}{2\lambda - 2}} |P|^{\frac{1}{2\lambda - 2}} - \lambda^2 2^{\frac{1}{\lambda - 1}} |P|^{\frac{3}{2\lambda - 2}} + o(|P|^{\frac{3}{2\lambda - 2}}).\]
Therefore,
\begin{align*}
T_- &= 2b \int_0^1 \frac{ds}{\sqrt{1-s^{2-2/\lambda}}} + O(b^3) \\
%&= 2^{1 + \frac{1}{2\lambda - 2}} |P|^{\frac{1}{2\lambda - 2}} \int_0^1 \frac{1}{\sqrt{1-s^{2-2/\lambda}}} \,ds + O(|P|^{\frac{3}{2\lambda - 2}}) \\
&= \frac{2^{1+\frac{1}{2\lambda-2}} \sqrt{\pi} \lambda \Gamma( \frac{\lambda}{2\lambda-2})}{\Gamma(\frac{1}{2\lambda-2})} |P|^{\frac{1}{2\lambda - 2}} + O(|P|^{\frac{3}{2\lambda - 2}}).
\end{align*}
We conclude that
\[\resizebox{\textwidth}{!}{$T = \begin{cases}
\displaystyle \pi \underbrace{- \frac{4 \lambda^{2\lambda - 2} \sqrt{\pi} \Gamma(\lambda) \tan(\pi\lambda)}{\Gamma(\lambda - \frac{1}{2})}}_{<\,0} |P| + o(|P|) & \quad \text{if}~ \lambda\in (1,\frac{3}{2}), \\ \\
\displaystyle \pi - \frac{9}{4} |P| \log(|P|^{-1}) + O(|P|) & \quad \text{if}~ \lambda = \frac{3}{2}, \\ \\
\displaystyle \pi \underbrace{-\frac{2^{1+\frac{1}{2\lambda-2}} \lambda \Gamma( \frac{2\lambda-1}{2\lambda-2} ) \Gamma( \frac{\lambda}{2\lambda-2})}{\sqrt{\pi}} \bigg( 1 - \sin \Big(\frac{\pi}{2\lambda - 2} \Big) \bigg)}_{<\,0} |P|^{\frac{1}{2\lambda-2}}+ O(|P|) & \quad \text{if}~ \lambda \in (\frac{3}{2}, 2).
\end{cases}$}\]
%\[ \underbrace{- \frac{2\sqrt{\pi} \lambda^{2\lambda - 2} \Gamma(\lambda) \tan(\lambda\pi)}{\Gamma(\lambda-\frac{1}{2})}}_{>\,0} |P| + o(|P|).\]
In any case, $T = T_+ + T_-$ is a decreasing function of $-P$ near $P = 0$. Furthermore, in the case $1 < \lambda \leq \frac{3}{2}$, as $-P$ increases near $0$, $T_+$ decreases at a rate of
$$\begin{cases} |P| & ~\text{if}~ 1 < \lambda < \frac{3}{2}, \\ |P|\log(|P|^{-1}) & ~\text{if}~ \lambda = \frac{3}{2}, \end{cases}$$
while $T_-$ increases at a rate of $|P|^{\frac{1}{2\lambda - 2}}$. Therefore, for $1 < \lambda \leq \frac{3}{2}$ and any positive integer $k$, the sum $T_+ + k T_-$ decreases as $-P$ increases near $0$.

\section{Relation to the Homogeneous Shear Flow} \label{sec:SHEAR}
In this section, we focus on the two-sink solution $\psi(\theta)$, which is obtained by gluing two copies each of $\psi_+$ and $\psi_-$ in the sequence $\psi_+$, $-\psi_-$, $\psi_+$, $-\psi_-$. The asymptotic behavior of the two-sink solution $r^\lambda \psi(\theta)$ as $\lambda\to 2$ is of interest as a means of comparison with the power-law shear flow $\Omega(x,y) = |y|^{-\alpha}$ (where $\alpha = 2-\lambda$). In particular, it appears at first sight that the two-sink approaches the power-law shear flow as $\lambda\to 2$ %or $\lambda\to 3/2$
(equivalently $\alpha\to 0$). %or $\alpha \to 1/2$).

Indeed, from the asymptotic expansion in the previous section we obtain
\begin{align*}
T &= \pi -\frac{2^{1+\frac{1}{2\lambda-2}} \lambda \Gamma( \frac{2\lambda-1}{2\lambda-2} ) \Gamma( \frac{\lambda}{2\lambda-2})}{\sqrt{\pi}} \bigg( 1 - \sin \Big(\frac{\pi}{2\lambda - 2} \Big) \bigg) |P|^{\frac{1}{2\lambda-2}} \\
& \qquad + \frac{2\sqrt{\pi} \lambda^{2\lambda - 2} \Gamma(\lambda) \tan((2-\lambda)\pi)}{\Gamma(\lambda-\frac{1}{2})} |P| + O(|P|^{1 + \frac{1}{2\lambda - 2}}).
\end{align*}
So we may estimate the value of the pressure $P^*$ at which the $\pi$-periodic solution is found. By using only the first three terms in the asymptotic expansion above, we have that $P^*$ is approximately:
\[ |P^*|^{\frac{2\lambda - 3}{2\lambda - 2}} \approx \frac{2^{\frac{1}{2\lambda - 2}} \Gamma( \frac{2\lambda-1}{2\lambda-2} ) \Gamma( \frac{\lambda}{2\lambda-2}) \Gamma(\lambda - \frac{1}{2})}{\pi \lambda^{2\lambda - 3} \Gamma(\lambda)} \frac{1 - \cos(\frac{2-\lambda}{2\lambda - 2} \pi)}{\tan((2-\lambda)\pi)} \]
% \[\resizebox{\textwidth}{!}{$-\pi ^{\frac{1}{3-2 \lambda }-1} \lambda ^{-2 \lambda +\frac{1}{3-2 \lambda }+1}
%    \left(\frac{2^{\frac{1}{2 (\lambda -1)}} \lambda  \left(\sin \left(\frac{\pi }{2
%    \lambda -2}\right)-1\right) \cot (\pi  \lambda ) \Gamma \left(1+\frac{1}{2 (\lambda
%    -1)}\right) \Gamma \left(\lambda -\frac{1}{2}\right) \Gamma \left(\frac{\lambda }{2
%    \lambda -2}\right)}{\Gamma (\lambda )}\right)^{\frac{1}{2 \lambda -3}+1}$}\]
This approximation for the critical pressure vanishes at $\lambda=2$.
Taking a series expansion in $\lambda$ about $\lambda = 2$, we have the following asymptotic behavior for $P^*$ as $\lambda \to 2^-$:
\begin{equation} \label{eq:ASYM}
|P^*(\lambda)| \approx \frac{\pi^2}{512}(2-\lambda)^2 \approx 0.0193 \alpha^2.
\end{equation}

These asymptotics are also verified by numerically evaluating $T$ and finding $P^*$. We believe that rigorously verifying Equation \eqref{eq:ASYM} as $\lambda \to 2^-$, is an interesting problem in its own right that should be quite challenging due to a degeneracy in the asymptotic formulas as $\lambda$ approaches $2$. 
\begin{figure}[H]
    \centering
    \includegraphics[width=0.8\linewidth]{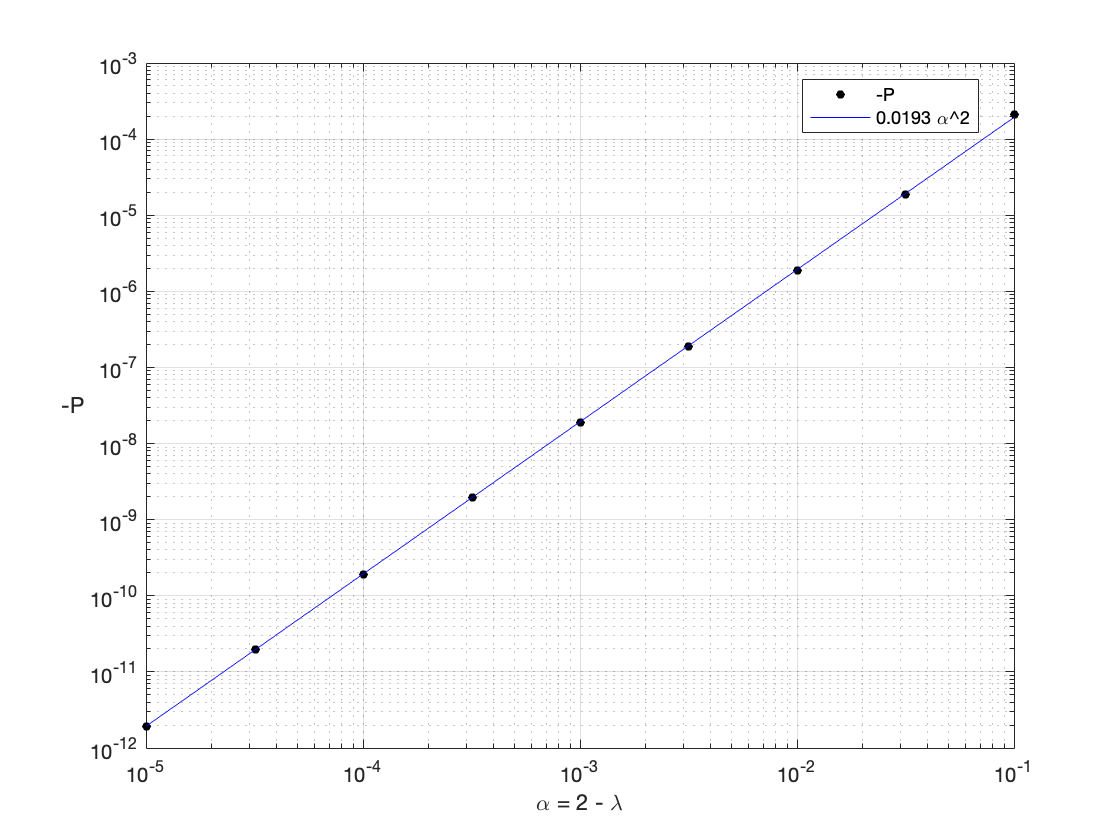}
    \caption{Log-log plot of $P^*$ by $\alpha$: data for this graph are located at Appendix C.}
    \label{fig:P2}
\end{figure}

% The asymptotic behavior as $\lambda\to 3/2$ is no longer algebraic, and $P^*$ is approximately:
% $$\frac{1}{3} 2^{\frac{2 (\lambda -1)}{2 \lambda -3}} \left(\lambda
%    -\frac{3}{2}\right)^{\frac{1}{2 \lambda -3}} e^{\frac{1}{2} \left(-2+\gamma -2 \pi -3
%    \psi ^{(0)}\left(\frac{3}{2}\right)\right)},$$
%    where $\gamma$ is the Euler-Mascharoni constant and $\psi^{(0)}(3/2)$ is the value of a certain polygamma special function.

\subsection{Convergence to the Shear Flow as Pressure Goes to Zero}
Since this "critical pressure" value approaches zero, it seems as if the two-sink solution should approach the power-law shear flow, which has identically zero pressure, as $\lambda\to 2$. 

To that end, we prove the following proposition, which shows convergence of the two-sink solution to the power-law shear flow at the velocity level when the critical pressure approaches zero.
\begin{proposition}
Let $\psi_P$ be a positive solution to the ordinary differential equation in Equation \eqref{eq:FIRSTORDER} with pressure $P<0$ on the interval $[0,T_+(P)]$ corresponding to the stream function of a two-sink solution. Let $\psi_S(x) = \lambda^{-\lambda}\sin^\lambda(\theta)$ be the explicit shear flow solution to Equation \eqref{eq:FIRSTORDER} on the interval $[0,\pi]$ (with zero pressure). Then we have
$$\lim_{P\to 0^-} \psi_{P} = \psi_S$$
in $C^{1,\gamma}$ on any compact subset of $[0,\pi)$ and any $\gamma < 1-1/\lambda$.
\end{proposition}
\begin{proof}
In this proof, we fix some $\lambda \in (1,2)$. In particular, when we say that a constant depends "only" on something else, we exclude from consideration the dependence on $\lambda$.

We recall that the positive function $\psi_P$ satisfies Equation \eqref{eq:FIRSTORDER}:
$$(\psi_P')^2 = -2P + \psi_P^{2-\frac{2}{\lambda}} - \lambda^2 \psi_P^2.$$
By the standard theory of ordinary differential equations, $\psi_P$ is smooth on $(0,T_+(P))$ and vanishes at the endpoints of the interval. 

Suppose first that there exists a sequence $P_n\to 0^-$ such that the solutions $\psi_n := \psi_{P_n}$ satisfy $\| \psi_n\|_{L^\infty}\to \infty$ as $n\to \infty$. Since $1<\lambda<2$ we know $2-2/\lambda<2$. Therefore, if $x_n$ is the number in the interval $[0,T_+(P_n)]$ at which the $L^\infty$ bound for each $\psi_n$ is saturated and if $\psi_n(x_n) := A_n$, we have
$$0\leq (\psi'_n)^2 = -2P_n +A_n^{2-2/\lambda}-\lambda^2 A_n^2 <0,$$
if we take $n$ large enough. This is a contradiction, so we may conclude that 
$$A:=\sup_{-1<P<0} \|\psi_P\|_{L^\infty} <\infty,$$
or equivalently that the solutions $\psi_P$ are uniformly bounded in $P$ near zero. Since $1<\lambda<2$, we know $2-2/\lambda>0$, so we immediately conclude that
$$\sup_{-1<P<0} \|\psi_P\|_{\text{Lip}} <L(A),$$
where $L(A)$ is a constant depending only on $A$. Equivalently, the family of functions $\psi_P$ have Lipschitz constant uniform in $P$ near zero. 

Since $\psi_P$ is Lipschitz, then $-2P + \psi_P^{2-\frac{2}{\lambda}} - \lambda^2 \psi_P^2$ is $C^{0,2-2/\lambda}$ H\"{o}lder regular, with H\"{o}lder norm depending only on $L(A)$, the uniform Lipschitz bound of $\psi_P$. In addition, $\pm \sqrt{-2P + \psi_P^{2-\frac{2}{\lambda}} - \lambda^2 \psi_P^2}$ is $C^{0,1-1/\lambda}$ H\"{o}lder regular, with H\"{o}lder norm depending only on $L(A)$, the uniform Lipschitz bound of $\psi_P$. We conclude that $\psi_P \in C^{1,1-1/\lambda}$, with H\"{o}lder norm depending solely on $A$, the uniform bound in $P$ on $\|\psi_P\|_{L^\infty}$.

By (a straightforward modification of) the Arzel\`{a}-Ascoli theorem, we conclude that there exists a sequence $P_m\to 0^-$ such that $\psi_m := \psi_{P_m}$ converges in $C^{1,\gamma}$ to some $C^{1,1-1/\lambda}$ function $\psi_\infty$ on any compact subset of $[0,\pi]$. If we prove that the limit is unique, taking any sequence $P\to 0^-$ makes $\psi_P$ converge to $\psi_\infty$. We also remark that since $\psi_P\geq 0$ for each $P$, we certainly have $\psi_\infty\geq 0$ on $[0,\pi$]. It remains to show that $\psi_\infty=\psi_S$.

By the convergence proven earlier, we know
$$(\psi_\infty')^2 = \psi_\infty^{2-\frac{2}{\lambda}} - \lambda^2 \psi_\infty^2.$$
Since $\psi_\infty\geq 0$, we may assume that $\psi_\infty = f^\lambda$ for some function $f\geq 0$. The ordinary differential equation satisfied by $f$ is precisely
$$(f')^2 + f^2 = \lambda^{-2}.$$
Since $f(0) = 0$ and $f\geq 0$, it is clear that $f(x)= \lambda^{-1}\sin(\theta)$. We conclude that $\psi_\infty=\psi_S$.
\end{proof}

\begin{remark}
    As noted earlier, the two-sink solution is more regular than the power-law shear flow, in a way that precludes analogous uniform convergence of the two-sink to the shear flow at the vorticity level.
\end{remark}

\section*{Appendix}
\renewcommand{\thesubsection}{\Alph{subsection}}
\renewcommand{\theequation}{A.\arabic{equation}}
\renewcommand{\thetheorem}{A.\arabic{theorem}}
\setcounter{equation}{0}
\setcounter{subsection}{0}
\setcounter{theorem}{0}
\subsection{The Mellin Transform and its Analytic Continuation}
Suppose $f:[0,\infty) \to \C$ is a function that for some $\alpha < \beta$,
\[f(t) = O(t^{-\alpha}) ~\text{as}~ t\to 0, \quad f(t) = O(t^{-\beta}) ~\text{as}~ t\to \infty.\]
Then, the Mellin transform $M[f;z]$ is defined on a strip $\alpha <\text{Re}\,z < \beta$ as
\[M[f;z] = \int_0^\infty t^{z-1} f(t) \,dt. \]
The Mellin transform $M[f;z]$ is a holomorphic function of $z$ on the strip $\alpha <\text{Re}\,z < \beta$. The inverse Mellin transform is given by
\[f(t) = \frac{1}{2\pi i} \int_{c-i\infty}^{c+i\infty} t^{-z} M[f;z] \,dz, \quad \alpha < c < \beta.\]
From this, it is possible to deduce the Parseval-Plancherel identity for the Mellin transform:
\[ \int_0^\infty f(t) g(t) \,dt = \frac{1}{2\pi i} \int_{c-i\infty}^{c+i\infty} M[f;z] M[g;1-z] \,dz. \]

Suppose $f(t) = \sum_{k=0}^\infty c_k t^{a_k}$ as $t\to 0^+$. Then, the Mellin transform $M[f;z]$ can be extended to a meromorphic function on a left half plane $\text{Re}\,z < \beta$ with simple poles at $z = -a_k$ for $k=1,2,\cdots$, and $\text{res}_{z=-a_k} M[f;z] = c_k$. Let $f(t) = \sum_{k=0}^{N-1} c_k t^{a_k} 1_{[0,1)}(t) + f_N(t)$. Then, on the strip $\alpha <\text{Re}\,z < \beta$, we get
\[M[f;z] = \sum_{k=0}^{N-1} c_k \int_0^1 t^{z+a_k-1} \,dt + M[f_N;z] = \sum_{k=0}^{N-1} \frac{c_k}{z+a_k} + M[f_N;z]. \]
The expression above is a meromorphic function on a strip $-a_N<\text{Re}\,z < \beta$ with simple poles at $z=-a_0,\cdots,-a_{N-1}$.

Similarly, if $f(t) = \sum_{k=0}^\infty d_k t^{-b_k}$ as $t\to\infty$. Then, the Mellin transform $M[f;z]$ can be extended to a meromorphic function on a right half plane $\text{Re}\,z > \alpha$ with simple poles at $z = b_k$ for $k=1,2,\cdots$, and $\text{res}_{z=b_k} M[f;z] = -d_k$.

\subsection{Asymptotic Expansions for Integrals}
Consider
\[I(\lambda) = \int_0^L f(t) g(\lambda t) \,dt\]
as $\lambda\to \infty$. From the Parseval-Plancherel identity for Mellin transform, we get
\[\int_0^L f(t) g(\lambda t) \,dt = \frac{1}{2\pi i} \int_{c-i\infty}^{c+i\infty} \lambda^{-z} M[f;1-z] M[g;z] \,dz.\]
Suppose
\[f(t) = \sum_{k=0}^\infty c_k t^{a_k} ~\text{as}~t\to 0^+, \quad g(t) = \sum_{l=0}^\infty d_l t^{-b_l} ~\text{as}~ t\to \infty. \]
If there is no $a_k$ and $b_l$ such that $a_k + 1 = b_l$, then simple poles of $M[f;1-z]$ and $M[g;z]$ does not overlap. By shifting the contour to right, we pick up residues and the following asymptotic expansion is obtained:
\[I(\lambda) = \sum_{k=0}^\infty c_k M[g;a_k+1] \lambda^{-a_k-1} + \sum_{l=0}^\infty d_l M[f;-b_l +1] \lambda^{-b_l}. \]
If $a_k + 1 = b_l = \gamma$ for some $k,l$, then $M[f;1-z] M[g,z]$ has a double pole at $z=-\gamma-1$. In this case, the corresponding term in the asymptotic expansion becomes
\[c_k d_l \lambda^{-\gamma} \log\lambda - \Big(\underset{z=-\gamma-1}{\text{res}} M[f;1-z] M[g,z] \Big) \lambda^{-\gamma}.\]
Now, consider
\[J(\lambda) = \int_0^L f(s) g(\lambda \phi(s)) \,ds\]
where $\phi:[0,L] \to [0,M]$ be strictly increasing bijection. Then, substituting $t = \phi(s)$ gives
\[J(\lambda) = \int_0^M \underbrace{\frac{f(s)}{\phi'(s)}}_{=\,F(t)} g(\lambda t) \,dt.\]
An asymptotic expansion can be obtained by the computation above, while
\[M[F;z] = \int_0^M t^{z-1} f(s) \frac{dt}{\phi'(s)} = \int_0^L \phi^{z-1}(s) f(s) \,ds. \]

\subsection{Numerical Data}

\begin{figure}[H]
{\setlength{\arraycolsep}{10pt}
\renewcommand{\arraystretch}{1.5}
\newcolumntype{?}{!{\vrule width 1pt}}
\[\begin{array}{?l|l|l?} \Xhline{1pt}
\lambda & - 2P^* & \alpha^{-2} P^* \\ \Xhline{1pt}
1.9 & 4.2333434 \cdot 10^{-4} & 0.0212 \\ \hline
1.96838 & 3.7657484 \cdot 10^{-5} & 0.0188 \\ \hline
1.99 & 3.7510429 \cdot 10^{-6} & 0.0188 \\ \hline
1.99684 & 3.7903849 \cdot 10^{-7} & 0.0190 \\ \hline
1.999 & 3.8277711 \cdot 10^{-8} & 0.0191 \\ \hline
1.99968 & 3.9359471 \cdot 10^{-9} & 0.0192 \\ \hline
1.9999 & 3.8507873 \cdot 10^{-10} & 0.0193 \\ \hline
1.999968 & 3.9461 \phantom{000} \cdot 10^{-11} & 0.0193 \\ \hline
1.99999 & 3.8547 \phantom{000} \cdot 10^{-12} & 0.0193 \\ \Xhline{1pt}
\end{array}\]}
\caption{Numerical investigation of $P^*$ for several values of $\lambda$ approaching $2$}
\end{figure}

Mathematica code for the generation of the vector plots in Figure \ref{fig:AFAR} and Figure \ref{fig:CLOSE} and Matlab code for the computation of the data above and the generation of Figure \ref{fig:P} and Figure \ref{fig:P2} may be found at the website of the second named author.

\begin{center}\texttt{https://web.math.princeton.edu/\textasciitilde mc3498/}.\end{center}

\end{document}